\documentclass[12pt,a4paper,oneside]{article}
\usepackage{amsfonts, amsmath, amssymb,latexsym, amsthm}
\usepackage{epsfig}
\usepackage{color}
\usepackage[all]{xy}

\usepackage[latin1]{inputenc}

\newtheorem{thm}{Theorem}[section]
\newtheorem{lemma}[thm]{Lemma}

\newtheorem{defi}[thm]{Definition}

\usepackage[%ps2pdf=true,
colorlinks]{hyperref}

\usepackage{memhfixc} % Solve problems between memoir and hyperref
\usepackage%[%figure,tableall]
{hypcap} % Correct a problem with hyperref
\hypersetup{
    bookmarksnumbered,
    pdfstartview={FitH},
    citecolor={blue},
    linkcolor={green},
    urlcolor={red},
    pdfpagemode={UseOutlines}
}

\usepackage{fancyhdr}

\title{Characterization of curves in $C^{(2)}$}
\author{Meritxell S\'aez}
\date{}
\begin{document}
\bibstyle{plain}

\maketitle

\begin{abstract}
In this paper we characterize the irreducible curves lying in $C^{(2)}$. We prove that a curve $B$ has a degree one morphism to $C^{(2)}$ with image a curve of degree $d$ with irreducible preimage in $C\times C$ if and only if there exists an irreducible smooth curve $D$ and morphisms from $D$ to $C$ and $B$ of degrees $d$ and $2$ respectively forming a diagram which does not reduce.  

{\parskip7pt
\noindent \textbf{Keywords}: Symmetric product, curve, irregular surface, curves in surfaces.}
\end{abstract}

\let\thefootnote\relax\footnote{The author has been partially supported by the Proyecto de Investigaci\'on MTM2012-38122-C03-02.}

\section{Introduction}
%In this paper we present a characterization of curves $\tilde{B}$ in $C^{(2)}$ from the existence of a curve $D$ in a diagram of morphisms of smooth curves with certain properties. 

Given a curve $B\subset C^{(2)}$, we define the \textbf{degree} of $B$ as the integer $d$ such that $C_P\cdot B=d$ where $C_P$ denotes the coordinate curve in $C^{(2)}$ with base point $P$. The curves of degree one in $C^{(2)}$ are completely characterized by the two results in \cite[Pg. 310, D-10]{ACGH} and \cite{Ciro}, where it is proven that a curve $B$ of degree one in $C^{(2)}$ different from a coordinate curve is smooth and it exists if and only if there exists a degree two morphism $f:C\rightarrow B$. Moreover, $B=\{f^{-1}(q)\ |\ q\in B\}\subset C^{(2)}$. 

In \cite{Chan} a different proof of this result is given. From this proof we remark that considering the curve $B\subset C^{(2)}$ as before, then, the preimage of $B$ by $\pi_C: C\times C\rightarrow C^{(2)}$ is isomorphic to $C$ through the projection onto the first factor.

%\begin{lemma}(\cite{ACGH})\label{csymrec}
%Suppose $f:C\rightarrow B$ is a degree $2$ morphism, with $B$ a smooth irreducible curve. Then we have a curve 
%\[
%\Sigma:=\{f^{-1}(q)\ |\ q\in B\}\subset C^{(2)}
%\] 
%of degree one such that $\Delta_C \cdot \Sigma=2g(C)-2-2(2g(B)-2)$, i.e. the degree of the ramification divisor of $f$.
%\end{lemma} 

%And conversely,

%\begin{lemma}(\cite{Chan})\label{Cn}
%Let $B\subset C^{(2)}$ be a degree one curve such that $B\not\subset Supp(C_P)$ $\forall P\in C$. Then, $B$ is smooth and there exists a degree $2$ map $f:C\rightarrow B$ such that $B=\{f^{-1}(Q)\ |\ Q\in B\}\subset C^{(2)}$.
%\end{lemma}

%\begin{rmk}\label{ncov}
%Therefore, if $f:C\rightarrow B$ is a degree $2$ morphism, then the curve $\{f^{-1}(q)\ |\ q\in B\}\subset C^{(2)}$ is a smooth curve isomorphic to $B$, so we can consider $B$ embedded in $C^{(2)}$. %Moreover, by the proof in \cite[Page 7]{Chan}, the preimage of $B$ by $\pi_C: C\times C\rightarrow C^{(2)}$ is isomorphic to $C$ through the projection onto the first factor.
%\end{rmk}

Let $\tilde{B}$ be an irreducible curve in $C^{(2)}$ different from a coordinate divisor. Let $B$ be its normalization and assume that there is no degree two morphism from $C$ to $B$. Then, since $C_P$ is ample in $C^{(2)}$, from the characterization of degree one curves we deduce that $\tilde{B}\cdot C_P\geq 2$. In this paper we present a characterization of curves with any degree. First of all we need the following definition: 

\begin{defi}\label{defred}
We say that a diagram of morphisms of curves 
\[
\xymatrix{
D \ar[d]_{(d:1)}\ar[r]^{(e:1)}& B\\
C &
}
\]
\textbf{reduces} if there exist curves $F$ and $H$ such that there exists a diagram
\[
\xymatrix{
D \ar@/_{10pt}/[dd]_{(d:1)}\ar[r]^{(e:1)} \ar@{.>}[d]^{(k:1)}& B \ar@{.>}[d]^{(k:1)}\\
F \ar@{.>}[d] \ar@{.>}[r]_{(e:1)} & H \\
C & 
}
\] 
with $k>1$, the upper square being a commutative diagram and the left vertical arrows giving a factorization of the original degree $d$ morphism.

When $k=d$ we will say that the diagram \textbf{completes}, and we will obtain a commutative diagram
\[
\xymatrix{
D \ar[d]_{(d:1)}\ar[r]^{(e:1)}& B \ar@{.>}[d]^{(d:1)}\\
C \ar@{.>}[r]_{(e:1)}& H.
}
\]
\end{defi}

Notice that when $d$ is a prime number both definitions are equivalent.

In this paper we prove 

\begin{thm}\label{charcurves}
Let $B$ be an irreducible smooth curve such that there are no non-trivial morphisms $B\rightarrow C$. A morphism of degree one from the curve $B$ to the surface $C^{(2)}$ exists, with image $\tilde{B}$ of degree $d$ if, and only if, there exists a smooth irreducible curve $D$ and a diagram 
\[
\xymatrix{
D \ar[d]_{(d:1)}\ar[r]^{(2:1)}& B\\
C &
}
\]
which does not reduce.
\end{thm}

If we consider the case $d=1$ we recover the results for degree one.

\vspace{5pt}
We prove the theorem in two steps, giving a separated proof for each implication(see Theorem \ref{const} and Theorem \ref{ext}). First, given a diagram 
\[
\xymatrix{
D \ar[d]_{(d:1)}^{g}\ar[r]^{(2:1)}_{f}& B\\
C &
}
\]
which does not reduce we find a curve in $C^{(2)}$ defined by it as the image by $g^{(2)}$ of the immersion of $B$ in $D^{(2)}$ given by $f$. We prove that $g^{(2)}|_{B}$ has degree one, and hence the curve in $C^{(2)}$ has normalization $B$ and the normalization map is precisely $g^{(2)}|_{B}$. Second, given a curve lying in $C^{(2)}$ we find a diagram defined by the curve, its preimage by $\pi_C$ and the projection on one factor of $C\times C$. We compute the degrees of the different maps and prove that this diagram does not reduce. 

In a following paper we are going to use this result to study and classify curves of degree two and some of degree three.

\textbf{Notation:} We work over the complex numbers. By curve we mean a complex projective reduced algebraic curve. Let $C$ be a smooth curve of genus $g\geq 2$, we put $C^{(2)}$ for its $2$nd symmetric product. We denote by $\pi_C:C \times C \rightarrow C^{(2)}$ the natural map, and $C_P \subset C^{(2)}$ a coordinate curve with base point $P\in C$. We put $\Delta_C$ for the main diagonal in $C^{(2)}$, and $\Delta_{C \times C}$ denotes the diagonal of the Cartesian product $C\times C$.

\section{Characterization}\label{Charsec}

We begin with a lemma that will simplify the rest of the exposition.

\begin{lemma}\label{digrdiag}
We consider a diagram of morphisms of smooth irreducible curves 
\[
\xymatrix{
D \ar[d]_{(d:1)}^{g}\ar[r]_{(2:1)}^{f}& B\\
C &.
}
\]
The image of $B\subset D^{(2)}$ (with the immersion given by the fibers of $f$) by the morphism $g^{(2)}$ is the diagonal $\Delta_C\subset C^{(2)}$ if and only if the morphism $g$ factorizes through the curve $B$ by $f$.
\end{lemma}

\begin{proof}
Let $i$ be the involution on $D$ that defines $f$, that is, the change of sheet. Since $B=\{x+y\ |\ f(x)=f(y)\}=\{x+i(x)\}\subset D^{(2)}$, then $Im(g^{(2)}|_{B})=\{g(x)+g(i(x))\}$. It is contained in the diagonal $\Delta_C$ if and only if $g(x)=g(i(x))$ for all $x\in D$, that is, if and only if $g$ factorizes through $B$ by $f$. 
\end{proof} 

In the following theorem, given a diagram that does not reduce we deduce the existence of a curve in $C^{(2)}$ naturally attached to it. 

\begin{thm}\label{const}
Assume that there exists a diagram of morphisms of smooth irreducible curves 
\[
\xymatrix{
D \ar[d]_{(d:1)}^{g}\ar[r]_{(2:1)}^{f}& B\\
C &
}
\]
which does not reduce and such that the morphism $g$ does not factorize through $B$ by $f$. Then, $g^{(2)}$ gives a degree one map $B\rightarrow C^{(2)}$ with reduced image a curve $\tilde{B}$ of degree precisely $d$. 

%Then there exists a degree one morphism from $B$ to $C^{(2)}$. We will call its image $\tilde{B}$. Moreover, the intersection of $\tilde{B}$ with a coordinate curve is precisely $d$.
\end{thm}

\begin{proof}
Consider a diagram as above and look at the induced morphism $D^{(2)}\stackrel{g^{(2)}}{\rightarrow}C^{(2)}$. As we have seen in the Introduction, we have an immersion $B\subset D^{(2)}$ as the set of pairs of points in $D$ with the same image by $f$. Then, we consider $D$ inside $D \times D$ as $\pi^{-1}_D(B)\cong D$, that is, ordered pairs of points with the same image by $f$.

Let $\tilde{B}=g^{(2)}(B)_{red}$, the reduced image curve in $C^{(2)}$, and consider the map $B\stackrel{(k:1)}{\rightarrow}\tilde{B}$ induced by $g^{(2)}$. We want to see that $k=1$.

Notice that by Lemma \ref{digrdiag} we can assume that $\tilde{B}$ is not $\Delta_C$. We know that $B\cdot D_P=1$, hence, 
\[
1={g^{(2)}}_*(B\cdot D_P)={g^{(2)}}_*(B)\cdot(\frac{1}{d}C_P) \Rightarrow {g^{(2)}}_*(B)\cdot C_P=d.
\]
In addition, since the map $B\stackrel{(k:1)}{\rightarrow} \tilde{B}$ is $g^{(2)}|_B$, we obtain that $d=(k\tilde{B})\cdot C_p$, and thus $\tilde{B}\cdot C_p=\frac{d}{k}$, that is, $k$ divides $d$. 

Assume by contradiction that $k>1$. 

Let $F$ be the preimage of $\tilde{B}$ by the morphism $\pi_C:C\times C \rightarrow C^{(2)}$. Then $F\rightarrow \tilde{B}$ has degree two and thus we obtain a diagram
\begin{equation}\label{diag}
\xymatrix@M+2pt{
D \times D \ar[r]^{\pi_D} \ar@/_20pt/[ddd]_{g\times g} & D^{(2)} \ar@/^20pt/[ddd]^{g^{(2)}}\\
D \ar@{_{(}->}[u] \ar[d] \ar[r]^{f}_{(2:1)}& B \ar[d]|(0.37){(k:1)} \ar@{_{(}->}[u] \\
F \ar@{_{(}->}[d] \ar[r]^{(2:1)}& \tilde{B}\ar@{_{(}->}[d]\\
C\times C \ar[r]^{\pi_C}& C^{(2)}.
}
\end{equation}
Observe that the exterior arrows form a commutative diagram, and hence, also the interior arrows give a commutative diagram. Thus, the morphism $D \rightarrow F$ has degree $k$. Now, the restriction to $D$ of $g\times g$ followed by the projection onto one factor of $C\times C$ is precisely $g:D\rightarrow C$ by construction. That is, we obtain the diagram
\[
\xymatrix{
D \ar@/_{10pt}/[dd]_{g} \ar[r]^{(2:1)} \ar[d]^{(k:1)}& B \ar[d]^{(k:1)}\\
F \ar[d] \ar[r]_{(2:1)} & \tilde{B} \\
C & .
}
\] 
Hence, the original diagram reduces, contradicting our hypothesis. 

Consequently, $k=1$ and thus we deduce that the curve $\tilde{B}$ has normalization $B$. 

Moreover, looking at diagram (\ref{diag}) we deduce that $D \stackrel{(1:1)}{\rightarrow} F$, that is, the preimage of $\tilde{B}$ by $\pi_C$ has normalization $D$,  and we will denote it by $\tilde{D}$. So we have:
\begin{equation}\label{diagcompl}
\xymatrix@M+2pt{
D \times D \ar[r]^{\pi_D} \ar@/_20pt/[ddd]_{g\times g} & D^{(2)} \ar@/^20pt/[ddd]^{g^{(2)}}\\
D \ar@{-->}@/_20pt/[ddd]_{g} \ar@{_{(}->}[u] \ar[d]|(0.37){\approx} \ar@{-->}[r]^{f}& B \ar[d]|(0.37){\approx} \ar@{_{(}->}[u] \\
\tilde{D} \ar@{_{(}->}[d] \ar[r]& \tilde{B}\ar@{_{(}->}[d]\\
C\times C \ar[r]^{\pi_C} \ar[d]^{pr}& C^{(2)}\\
C&
}
\end{equation}
where the dashed arrows show the original diagram.
\end{proof}

Conversely, we have also a theorem in the opposite direction, from the existence of curves in $C^{(2)}$ we deduce the existence of diagrams which do not reduce.

\begin{thm}\label{ext}
Given an irreducible curve $\tilde{B}$ lying in $C^{(2)}$ with degree $d$, let $B$ be its normalization, and assume that there are no non trivial morphisms $B\rightarrow C$. Then, there exists a smooth irreducible curve $D$ and a diagram 
\[
\xymatrix{
D \ar[d]_{(d:1)}\ar[r]^{(2:1)}& B\\
C &
}
\]
which does not reduce.
\end{thm}

\begin{proof}
First of all, we observe that $\tilde{B}$ is not the diagonal in $C^{(2)}$ because we are assuming that there are no morphisms from $B$ to $C$. 

Let $\tilde{D}=\pi_C^*(\tilde{B})\in Div(C\times C)$ and $D$ its normalization. We notice that with our hypothesis $\tilde{D}$ is irreducible. Indeed, otherwise, one of its components would have as normalization the curve $B$, because we have a $(2:1)$ morphism from $\tilde{D}$ to $B$, and since $\tilde{D}\subset C\times C$ we would obtain a non trivial morphism from $B$ to $C$ contradicting our hypothesis.

Now, we are going to compute the degree of $\tilde{D}\rightarrow C$, given by the projection onto one factor:
\[
\begin{array}{c}
\tilde{D}\cdot (C\times P + P\times C)={\pi_C}_*({\pi_C}^*(\tilde{B})\cdot {\pi_C}^*(C_P))= \\
\tilde{B}\cdot {\pi_C}_*{\pi_C}^*(C_P)=2\tilde{B}\cdot C_P=2d.
\end{array}
\]
And therefore, since $\tilde{D}$ is symmetric with respect to the involution $(x,y) \rightarrow (y,x)$ by construction, $\tilde{D}\cdot (C\times P)=d$, and so, the degree of the morphism on $C$ is precisely $d$. In this way, we have a diagram \vspace{-10pt}
\[
\begin{array}{ccc}
{\xymatrix{
\tilde{D}\ar[r]^{(2:1)}\ar[d]_{(d:1)}&\tilde{B}\\
C&
}} & {\begin{array}{c} \\ \textrm{and taking their normalizations} \\ \textrm{we obtain a diagram of} \\ \textrm{morphisms of smooth curves} \end{array}}&
{\xymatrix{
D\ar[r]^{(2:1)}\ar[d]_{(d:1)}&B\\
C&. \vspace{-10pt}
}}
\end{array}
\] 

We call $f:D \rightarrow B$ the map coming from $\pi_C|_{\tilde{D}}$ and $g:D \rightarrow C$ the map coming from the projection onto one factor of $C\times C$.

Let $\alpha$ be the degree one morphism induced in $B$ by the immersion of $\tilde{B}$ in $C^{(2)}$. Since we have $D \stackrel{(2:1)}{\longrightarrow}B$, as we have seen in the Introduction there exists an immersion of $B$ in $D^{(2)}$ as pairs of points with the same image by this morphism. 

Since $D\stackrel{(1:1)}{\longrightarrow} \tilde{D} \subset C\times C$ we can consider that a general point in $D$ is a pair $(x,y)$ with $x,y\in C$. Moreover, since $D\rightarrow B$ is induced by $\pi_C|_{\tilde{D}}$, a general fiber of $D\rightarrow B$ will be two points $(x,y)$ and $(y,x)$. Hence, we can write a general point of $B\subset D^{(2)}$ as $(x,y)+(y,x)$. 

Now, we consider the restriction to $B\subset D^{(2)}$ of $g^{(2)}$. By construction this morphism is precisely $\alpha$ and therefore the image is the original $\tilde{B}$. In particular, $g^{(2)}|_B$ is generically of degree one.  

We are going to see that the diagram does not reduce by contradiction: Assume that there exist curves $F$ and $H$, and a diagram \vspace{-10pt}
\[
\xymatrix{
D \ar@/_{10pt}/[dd]_{g}\ar[r]^{f} \ar[d]^{h}& B \ar[d]^{r}\\
F \ar[d]_{l} \ar[r]^{s} & H \\
C & \vspace{-10pt}
}
\] 
as in Definition \ref{defred}. Then, as we have seen in the Introduction, the fibers of $s$ give a curve isomorphic to $H$ inside $F^{(2)}$. Hence, we have
\[
\xymatrix@M+2pt{
& & \tilde{B}\ar@{_{(}->}[d] \\
B \ar@{^{(}->}[r] \ar[d] \ar@/^12pt/[rru]^{\alpha}&	D^{(2)} \ar[r]^{g^{(2)}} \ar[d]^{h^{(2)}} & C^{(2)}   \\
H	\ar@{^{(}->}[r] \ar@/_80pt/[rruu] & F^{(2)}\ar[ur]_{l^{(2)}} & \\
&&
}
\]

By definition, the image of $B\subset D^{(2)}$ by $h^{(2)}$ is $H\subset F^{(2)}$, that is, the embedding of $H$ in $F^{(2)}$ given by $s$, and we know that $l \circ h=g$ so $l^{(2)}\circ h^{(2)}=g^{(2)}$, hence 
\[
{g^{(2)}}|_B:
\xymatrix@C+15pt{
 B \ar[r]^{h^{(2)}=r} \ar@/^20pt/[rr]^{(1:1)} & H \ar[r]^{l^{(2)}}& \tilde{B} 
}
\]
thus $r$, as well as $h$, have degree one. Consequently, our diagram does not reduce (see Definition \ref{defred}). 
\end{proof}

We observe that we could change the hypothesis of the non existence of morphisms from $B$ to $C$ by assuming that $\tilde{B}$ is not the diagonal and that $\pi^{-1}_C(\tilde{B})$ is irreducible.

Putting these two theorems together we find the characterization of curves in the symmetric square $C^{(2)}$ previously stated in Theorem \ref{charcurves}.

%\noindent \textbf{Acknowledgments}

\bibliographystyle{alpha}
\bibliography{bib}
\end{document}